\documentclass{amsart}

\usepackage[T1]{fontenc}
\usepackage[utf8]{inputenc}
\usepackage{mathscinet}
\usepackage{hyperref}
\usepackage{url}
\usepackage[arrow,cmtip,matrix]{xy}

\newcommand{\NoOp}[1]{}

\newcommand\bmu{\boldsymbol{\mu}}

\newcommand\cE{\mathcal{E}}

\newcommand\fppf{\mathit{fppf}}

\newcommand\Gm{\mathbf{G}_{\mathrm{m}}}
\newcommand\Gmover[1]{\mathbf{G}_{\mathrm{m},#1}}
\renewcommand\H{\mathrm{H}}
\newcommand\HH{\H_{\mathrm{H}}}  

\newcommand\op{{\mathrm{op}}}
\newcommand\Sel{\mathrm{Sel}}
\newcommand\Tors{\mathcal{T}}
\newcommand\Z{\mathbf{Z}}

\newcommand\blank{\kern12pt}
\newcommand\isom{\stackrel\sim\longrightarrow}
\newcommand\morphism[1]{\stackrel{#1}\longrightarrow}
\newcommand\tensor[1]{\mathbin{\mathop{\otimes}\limits_{#1}}}
\newcommand\ses[3]{1\longrightarrow #1\longrightarrow #2
\longrightarrow #3\longrightarrow1}

\DeclareMathOperator\Ext{Ext}
\DeclareMathOperator\Hom{Hom}
\DeclareMathOperator\id{id}
\DeclareMathOperator\Pic{Pic}
\DeclareMathOperator\Spec{Spec}

\newtheorem{theorem}{Theorem}[section]

\newtheorem{lemma}[theorem]{Lemma}

\theoremstyle{definition}
\newtheorem{assumption}[theorem]{Assumption}
\newtheorem{definition}[theorem]{Definition}
\newtheorem{example}[theorem]{Example}

\theoremstyle{remark}
\newtheorem{remark}[theorem]{Remark}

\numberwithin{equation}{section}

\title{Extensions and torsors for finite group schemes}

\author{Peter Bruin}

\begin{document}

\begin{abstract}
We give an explicit description of the category of central extensions
of a group scheme by a sheaf of Abelian groups.  Based on this, we
describe a framework for computing with central extensions of finite
commutative group schemes, torsors under such group schemes and groups
of isomorphism classes of these objects.
\end{abstract}

\maketitle

\section{Introduction}

Let $G$ be a finite locally free group scheme over a scheme~$S$.  We
describe the category of central extensions of $G$ by a commutative
\emph{fppf} group scheme $F$ affine over~$S$, and for $G$ commutative
also the category of $G$-torsors over~$S$, in a way that is suitable
for explicit calculations.

Under certain computational assumptions (which are fulfilled, for
example, if $K$ is a number field, $S$ is the spectrum of the ring of
$\Sigma$-integers in~$K$ with $\Sigma$ a finite set of places of~$K$,
and $F$ is itself finite locally free or $F=\Gm$), we give algorithms
for computing
\begin{itemize}
\item the extension class group $\Ext_S(G,F)$, i.e.\ the group of
  isomorphism classes of central extensions of $G$ by~$F$,
\end{itemize}
and for commutative $G$ also
\begin{itemize}
\item the subgroup of $\Ext_S(G,F)$ classifying commutative
  extensions, and
\item the torsor class group $\H^1(S_\fppf,G)$, i.e.\ the group of
  isomorphism classes of $G$-torsors over~$S$.
\end{itemize}
These algorithms ultimately reduce the problem to the computation of
unit groups and Picard groups of certain finite locally free
$S$-schemes.

\subsubsection*{Outline of the paper}

In \S\ref{sec:extension-data}, we introduce some preliminary notions
and define \emph{$F$-extension data} on a group scheme $G$ over a
scheme~$S$, where $F$ is a sheaf of Abelian groups $F$ on $S_\fppf$.
In \S\ref{sec:correspondence}, we show that central extensions of $G$
by~$F$ are classified by $F$-extension data on~$G$
(Theorem~\ref{th:equivalence}), and construct an exact sequence
relating the group $\Ext_S(G,F)$ to various cohomology groups
(Theorem~\ref{th:exact-seq}).  For $G$ finite locally free and
commutative, we show in \S\ref{sec:Ext-to-H1} how to use
$\Gm$-extension data on~$G$ to describe $G^*$-torsors over~$S$, making
a theorem of Chase explicit.  Finally, in \S\ref{sec:computational},
we show how the theory developed in this paper leads to algorithms for
computing the above objects in practice for a finite locally free
commutative group scheme over suitable base schemes, and we describe a
connection between our results and algorithms for computing Selmer
groups of elliptic curves.

\section{Extension data on a group scheme}

\label{sec:extension-data}

Let $S$ be a scheme, and let $G$ be a group scheme over~$S$.
We denote the group operation, identity and inverse morphisms of~$G$
by $m\colon G\times G\to G$, $e\colon S\to G$ and
$\iota\colon G\to G$.

Let $F$ be a sheaf of Abelian groups on $S_\fppf$.
We use multiplicative notation for~$F$ since important examples are
the multiplicative group or the group of $n$-th roots of unity for
some $n\ge1$.

For every $S$-scheme $X$, let $\Tors_F(X)$ be the category of
$F$-torsors on~$X_\fppf$.  We write $T\otimes T'$ for the contracted
product of two $F$-torsors $T$ and~$T'$, and $T^\vee$ for the dual of
an $F$-torsor~$T$.

\subsection{Some simplicial definitions}

\label{simplicial}

For all $k\ge0$, we write $G^k$ for the $k$-fold fibre power of $G$
over~$S$.  We number the factors by $\{0,1,\ldots,k-1\}$ and write
$p_i$ for the projection on the $i$-th coordinate.

The morphisms
\[
p_0, p_1, m\colon G^2\to G
\]
give rise to a group homomorphism
\begin{align*}
d^1\colon F(G)&\longrightarrow F(G^2)\\
x&\longmapsto (p_1^*x)(m^*x)^{-1}(p_0^*x).
\end{align*}
The above morphisms also give rise to functors
\[
p_0^*, p_1^*, m^*\colon\Tors_F(G)\longrightarrow\Tors_F(G^2)
\]
and hence a functor
\begin{align*}
\delta^1\colon\Tors_F(G)&\longrightarrow\Tors_F(G^2)\\
T&\longmapsto p_1^*T\otimes(m^*T)^\vee\otimes p_0^*T.
\end{align*}
Similarly, we consider the morphisms
\[
p_{0,1},p_{1,2},m_{0,1},m_{1,2}\colon G^3\to G^2
\]
defined by
\[
p_{0,1} = {\id_G}\times p_0,\quad
p_{1,2} = p_1\times{\id_G},\quad
m_{0,1} = m\times{\id_G},\quad
m_{1,2} = {\id_G}\times m.
\]
These give rise to a group homomorphism
\begin{align*}
d^2\colon F(G^2)&\longrightarrow F(G^3)\\
x&\longmapsto (p_{1,2}^*x)(m_{0,1}^*x)^{-1}
(m_{1,2}^*x)(p_{0,1}^*x)^{-1}.
\end{align*}
The above morphisms also give rise to functors
\[
p_{0,1}^*, p_{1,2}^*, m_{0,1}^*, m_{1,2}^*\colon
\Tors_F(G^2)\longrightarrow\Tors_F(G^3)
\]
and hence a functor
\begin{align*}
\delta^2\colon\Tors_F(G^2)&\longrightarrow\Tors_F(G^3)\\
T&\longmapsto p_{1,2}^*T\otimes(m_{0,1}^*T)^\vee
\otimes m_{1,2}T\otimes (p_{0,1}^*T)^\vee.
\end{align*}

The morphisms $d^1$ and~$d^2$ are part of the \emph{Hochschild
complex}
\begin{equation}
\label{eq:hochschild-complex}
F(S)\morphism{d^0} F(G)\morphism{d^1} F(G^2)\morphism{d^2}
F(G^3)\morphism{d^3}\cdots,
\end{equation}
whose cohomology groups are the \emph{Hochschild cohomology groups}
of~$G$ with coefficients in~$F$.

For every $F$-torsor $T$ on~$G$, there is a canonical trivialisation
\[
\kappa_T\colon F_{G^3}\isom\delta^2(\delta^1 T).
\]

\subsection{Extension data}

The following definition forms the basis for our computational
framework for group scheme extensions.

\begin{definition}
Let $G$ be a group scheme over a scheme~$S$, and let $F$ be a sheaf of
Abelian groups on $S_\fppf$.
An \emph{$F$-extension datum} on~$G$ is a pair $(T,\tau)$ where $T$ is
an $F$-torsor on~$G$ and
\[
\tau\colon F_{G^2}\isom\delta^1 T
\]
is an isomorphism of $F$-torsors on $G^2$ such that the triangle
\begin{equation}
\label{eq:triangle}
\xymatrix{
F_{G^3} \ar[r]^\sim \ar[dr]^\sim_{\kappa_T}&
\delta^2 F_{G^2} \ar[d]_\sim^{\delta^2\tau}\\
& \delta^2(\delta^1 T)}
\end{equation}
commutes.  Given two $F$-extension data $(T,\tau)$ and $(T',\tau')$
on~$G$, an \emph{isomorphism} from $(T,\tau)$ to $(T',\tau')$ is an
isomorphism
\[
\phi\colon T\isom T'
\]
of $F$-torsors on~$G$ such that the triangle
\begin{equation}
\label{eq:triangle-iso}
\xymatrix{
  F_{G^2} \ar[r]^{\tau}_\sim \ar[dr]^\sim_{\tau'}&
  \delta^1T\ar[d]^{\delta^1\phi}_\sim\\
  & \delta^1T'}
\end{equation}
commutes.  The \emph{groupoid of $F$-extension data on~$G$}, denoted
by $\cE(G,F)$, is the groupoid in which the objects are the extension
data for $(G,F)$ and the isomorphisms are as above.
\end{definition}

Note that the contracted product makes $\cE(G,F)$ into a symmetric
monoidal category.  The neutral object is $(F_G,\tau_0)$ where
$\tau_0\colon F_{G^2}\isom\delta^1 F_G$ is the canonical isomorphism.
In particular, we have an Abelian group of isomorphism classes of
objects of $\cE(G,F)$.

\section{Correspondence between extension data and group scheme extensions}

\label{sec:correspondence}

\subsection{The extension datum defined by a group scheme extension}

\label{ext-to-datum}

From now on, we assume that the sheaf $F$ is representable,
\emph{fppf} and affine over~$S$.  Then every $F$-torsor over an
$S$-scheme $X$ is representable, \emph{fppf} and affine over~$X$ (see
for example \cite[\S17]{Oort}).

\begin{remark}
The assumption that $F$ is representable, \emph{fppf} and affine
over~$S$ is made for convenience and can probably be weakened or
removed.
\end{remark}

Consider a central extension
\[
1\longrightarrow F\morphism{j}E\morphism{q}G\longrightarrow 1
\]
of sheaves of groups on $S_\fppf$.  Then $q$ makes $E$ into an
$F$-torsor over~$G$, so $E$ is representable.  Let $m_E\colon
E\times_S E\to E$ and $\iota_E\colon E\to E$ be the multiplication and
inverse morphisms.

We have a commutative diagram
\[
\xymatrix{
  E\times_S E \ar[r]^{m_E} \ar[d]_{q\times q}& E \ar[d]^q\\
  G\times_S G\ar[r]^m& G.}
\]
There is a canonical morphism
\[
E\times_S E\to p_0^*E \otimes p_1^*E
\]
of $G^2$-schemes.  It is straightforward to check that $m_E$ induces
an $F$-equivariant morphism
\[
\nu_E\colon  p_0^*E \otimes p_1^*E \to m^*E,
\]
which is automatically an isomorphism because both sides are
$F$-torsors.  We therefore obtain a trivialisation
\[
\tau_E\colon F_{G^2} \isom \delta^1 E.
\]
By associativity of the group operation of~$E$, we have a commutative
diagram
\[
\xymatrix{
  E\times_S E\times_S E \ar[r]^{\mkern28mu m_E\times\id} \ar[d]_{{\id}\times m_E}&
  E\times_S E \ar[d]^{m_E}\\
  E\times_S E \ar[r]_{m_E}& E}
\]
lying over the corresponding diagram for~$G$.  The commutativity of
this diagram is equivalent to the statement that the isomorphisms
\begin{align*}
p_0^*E\otimes p_1^*E\otimes p_2^*E&\isom
p_{0,1}^*(p_0^*E\otimes p_1^*E)\otimes p_2^*E
\morphism{p_{0,1}^*\nu_E\otimes\id}
p_{0,1}^*m^*E\otimes p_2^*E\\&\isom
m_{0,1}^*(p_0^*E\otimes p_1^*E)
\morphism{m_{0,1}^*\nu_E}
m_{0,1}^*m^*E
\end{align*}
and
\begin{align*}
p_0^*E\otimes p_1^*E\otimes p_2^*E&\isom
p_0^*E\otimes p_{1,2}^*(p_0^*E\otimes p_1^*E)
\morphism{{\id}\otimes p_{1,2}^*\nu_E}
p_0^*E\otimes p_{1,2}^*m^*E\\&\isom
m_{1,2}^*(p_0^*E\otimes p_1^*E)
\morphism{m_{1,2}^*\nu_E}
m_{1,2}^*m^*E
\end{align*}
coincide with each other under the canonical identification of
$m_{0,1}^*m^*E$ and $m_{1,2}^*m^*E$ given by the associativity
of~$G$.  This is in turn equivalent to the commutativity of the
triangle \eqref{eq:triangle} for $T=E$ and $\tau=\tau_E$.  We conclude
that $(E,\tau_E)$ is an $F$-extension datum on~$G$.

\subsection{The group scheme extension defined by an extension datum}

\label{datum-to-ext}

Conversely, let $(T,\tau)$ be an $F$-extension datum on~$G$.  As
remarked above, $T$ is representable.  Let $q\colon T\to G$ be the
structure map.  We will use $\tau$ to make $T$ into a group scheme
over~$S$ equipped with a homomorphism $j_\tau\colon F\to T$ such that
$T$ becomes a central extension of $G$ by~$F$.  This extends the
well-known construction of a central extension of an abstract
group~$\Gamma$ by an Abelian group~$A$ from a 2-cocycle.

The trivialisation~$\tau$ induces an isomorphism $\nu_\tau\colon
p_0^*T\otimes p_1^* T\to m^*T$ of $F$-torsors on $G\times G$, and
hence a morphism
\begin{equation}
\label{eq:m-tau}
m_\tau\colon T\times_S T\to T
\end{equation}
of $S$-schemes fitting in a commutative diagram
\[
\xymatrix{
  T\times_S T \ar[r]^{m_\tau} \ar[d]_{q\times q}& T \ar[d]^q\\
  G\times_S G\ar[r]^m& G.}
\]
By the argument in \S\ref{ext-to-datum}, $m_T$ is an associative
operation on~$T$.  Pulling back $\tau$ via the morphism
\[
(e,e)\colon S\to G^2
\]
gives a trivialisation
\begin{equation}
\label{eq:tau-ee}
\tau(e,e)\colon F\isom e^* T
\end{equation}
and hence an $F$-equivariant closed immersion $j_\tau\colon F\to T$
fitting in a commutative diagram
\[
\xymatrix{
F \ar[r]^{j_\tau} \ar[d] & T \ar[d]^q\\
S \ar[r]^e & G\rlap.}
\]

Next, pulling back $\tau$ via the morphisms
\[
i_0 = {\id}\times e\colon G\to G^2,\quad
i_1 = e\times{\id}\colon G\to G^2
\]
gives trivialisations
\[
\tau(\blank,e)\colon F_G\isom p^*e^* T,\quad
\tau(e,\blank)\colon F_G\isom p^*e^* T
\]
It is straightforward to check that pulling back the commutative
triangle~\eqref{eq:triangle} by the morphisms
\[
{\id}\times e\times e,
e\times e\times{\id}
\colon G\to G^3
\]
implies that both $\tau(\blank,e)$ and $\tau(e,\blank)$
are equal to the pull-back of $\tau(e,e)$ via~$p$.
From this it follows that the compositions
\[
F\times_S T\morphism{j_\tau\times{\id}} T\times_S T\morphism{m_\tau}T
\quad\text{and}\quad
T\times_S F\morphism{{\id}\times j_\tau} T\times_S T\morphism{m_\tau}T
\]
agree with the maps given by the $F$-torsor structure on~$T$; in
particular, $j_\tau$ sends the identity section of~$F$ to a two-sided
identity section for~$m_\tau$, and $j_\tau$ is a homomorphism of
monoid schemes.

Finally, pulling back $\tau$ via the morphisms
\[
(\iota,\id)\colon G\to G^2,\quad
(\id,\iota)\colon G\to G^2
\]
gives trivialisations
\begin{align*}
(\id,\iota)^*\tau\colon F_G&\isom
\iota^* T\otimes(p^*e^* T)^\vee\otimes T,\\
(\iota,\id)^*\tau\colon F_G&\isom
T\otimes(p^*e^* T)^\vee\otimes\iota^* T.
\end{align*}
Via the isomorphism $\tau(e,e)$ from~\eqref{eq:tau-ee},
we obtain isomorphisms
\[
\phi_0,\phi_1\colon T^\vee\isom\iota^* T
\]
of $F$-torsors on~$G$.  Composing these with the canonical isomorphism
$T\isom T^\vee$ of $G$-schemes (not an isomorphism of $F$-torsors
since it inverts the action of~$F$) gives isomorphisms
$T\isom\iota^*T$ over~$G$, hence automorphisms $T\isom T$ lying over
the automorphism $\iota$ of~$G$.  By construction, these are left
inverse and right inverse operations on~$T$, and they are equal by the
associativity of~$m_\tau$.  Thus we have equipped $T$ with the
desired structure of central extension of $G$ by~$F$.

\subsection{An equivalence of categories and an exact sequence}

\label{equiv-exact}

\begin{theorem}
\label{th:equivalence}
Let $G$ be a group scheme over a scheme~$S$, and let $F$ be a sheaf of
Abelian groups on~$G$.  Assume that $F$ is representable, $\fppf$ and
affine over~$S$.  Then there is a canonical equivalence of categories
between $\cE(G,F)$ and the category of central extensions of $G$
by~$F$.
\end{theorem}

\begin{proof}
By the arguments in \S\ref{ext-to-datum} and \S\ref{datum-to-ext}, the
possible group scheme structures on an $F$-torsor $T$ on~$G$
correspond bijectively to the trivialisations $\tau\colon
F_{G^2}\isom\delta^1 T$ making the triangle~\ref{eq:triangle} commute.
This gives a canonical bijection between central extensions of $G$
by~$F$ and objects of $\cE(G,F)$.

Given extension data $(T,\tau)$ and $(T',\tau')$, an isomorphism
between the corresponding extensions is an isomorphism $T\to T'$ of
$F$-torsors on~$G$ that is compatible with the group structures
$m_\tau$ and $m_{\tau'}$ as well as the inclusions $j_\tau\colon
F\to T$ and $j_{\tau'}\colon F\to T'$.  The first condition
corresponds to the commutativity of the
diagram~\eqref{eq:triangle-iso}, and the second condition follows from
the first by the definition of $j_\tau$ and~$j_{\tau'}$.
\end{proof}

Via the equivalence of Theorem~\ref{th:equivalence}, we will identify
the group $\Ext_S(G,F)$ of isomorphism classes of central extensions
of $G$ by~$F$ with the group of isomorphism classes of objects of
$\cE(G,F)$.  Using this identification, we will embed $\Ext_S(G,F)$
into an exact sequence that will allow us to compute this group in
various cases.

We first consider $F$-extension data $(T,\tau)$ such that $T$ is a
trivial $F$-torsor.  Given an element $u\in F(G^2)$ with $d^2u=1$, one
obtains an extension datum $(F,\tau)$ such that $\tau$ corresponds to
multiplication by~$u$ under the canonical identification
$F_{G^2}\isom\delta^1 F$.  This gives an injective homomorphism
\[
\HH^2(G,F)\longrightarrow\Ext_S(G,F)
\]
whose image is the group of isomorphism classes of extension data
$(T,\tau)$ such that the $F$-torsor $T$ is trivial; see \cite[expos\'e
  XVII, appendice~I]{SGA3II}.

Next, we define $K(G,F)$ to be the kernel of the map
\begin{align*}
d^1\colon\H^1(G_\fppf,F)&\longrightarrow\H^1(G^2_\fppf,F)
\end{align*}
induced by the functor~$\delta^1$.  By our identification of
$\Ext_S(G,F)$ with the group of isomorphism classes of objects of
$\cE(G,F)$, there is a canonical group homomorphism
\[
\Ext_S(G,F)\longrightarrow K(G,F)
\]
sending the class of $(T,\tau)$ to the class of~$T$.

Furthermore, we construct a homomorphism
\begin{equation}
\label{eq:trg}
\xi_{G,F}\colon K(G,F)\longrightarrow\HH^3(G,F)
\end{equation}
as follows.  Let $x$ be an element of $K(G,F)$, represented by an
$F$-torsor $T$ on~$G$ such that $\delta^1 T$ is trivial.  Choose a
trivialisation
\[
\tau\colon F_{G^2}\isom\delta^1 T.
\]
We then define $u_\tau$ to be the element of $F(G^3)$ such that the
composition
\[
F_{G^3}\isom \delta^2 F_{G^2} \morphism{\delta^2\tau} \delta^2(\delta^1 T)
\morphism{\kappa_T^{-1}} F_{G^3}
\]
equals multiplication by~$u_\tau$.  Then we have $d^2 u=1$, and we
define $\xi_{G,F}(x)$ to be the class of $u_\tau$ in $\HH^2(G,F)$.
Since a different choice for $T$ or~$\tau$ changes $u_\tau$ by an
element in the image of $d^2$, the map $\xi_{G,F}$ is a well-defined
group homomorphism.

\begin{theorem}
\label{th:exact-seq}
There is an exact sequence
\[
1\longrightarrow\HH^2(G,F)\longrightarrow\Ext_S(G,F)
\longrightarrow K(G,F)\morphism{\xi_{G,F}}\HH^3(G,F).
\]
\end{theorem}

\begin{proof}
We define a sequence using the maps constructed above.  Exactness at
$\HH^2(G,F)$ and $\Ext_S(G,F)$ follows from the above arguments.  It
remains to show exactness at $K(G,F)$.  Given an $F$-extension datum
$(T,\tau)$ on~$G$, the element $u_\tau$ equals 1 by the commutativity
of~\eqref{eq:triangle}, so the class of $T$ is in the kernel
of~$\xi_{G,F}$.  Conversely, let $x\in K(G,F)$ be in the kernel of $\xi_{G,F}$.
Choosing $T$ and~$\tau$ as in the construction of $\xi_{G,F}$, the element
$u_\tau\in F(G^3)$ is then in the image of~$d^2$, say $u_\tau=d^2 y$
with $y\in F(G^2)$.  Dividing $\tau$ by~$y$, we obtain a
trivialisation $\tau'\colon F_{G^2}\isom\delta^1 T$ such that
$u_{\tau'}=1$, so the diagram~\eqref{eq:triangle} for~$\tau'$
commutes; therefore, $(T,\tau')$ is in $\cE(G,F)$ and maps to $x\in
K(G,F)$.  It follows that the kernel of~$\xi_{G,F}$ equals the image of
$\Ext_S(G,F)$ in $K(G,F)$.
\end{proof}

\begin{remark}
It is well known that extensions of an abstract group $\Gamma$ by a
$\Gamma$-module $A$ are classified by the group $\H^2(\Gamma,A)$ where
$\H^i(\Gamma,\blank)$ is the $i$-th derived functor of the functor of
$\Gamma$-invariants.  For a group scheme $G$ over a scheme~$S$, there
is a functor of $G$-invariants defined for a sheaf $F$ of $G$-modules
by $\H^0(G,F)=F^G(S)$, where $F^G$ is the sheaf of $G$-invariants.
There is a homomorphism $\Ext_S(G,F)\to\H^2(G,F)$, but this is part of
a long exact sequence and is in general not an isomorphism; see
Demazure and Gabriel \cite[III.6.3.1]{Demazure-Gabriel}.
\end{remark}

\begin{remark}
It is tempting to try to construct the exact sequence of
Theorem~\ref{th:exact-seq} as the exact sequence of low-degree terms
arising from a spectral sequence with $E_1^{p,q}$-terms $\H^q(G^p,F)$
for $p\ge1$ and $q\ge0$.  The author has so far been able to construct
such a spectral sequence only in the case where $G$ is finite over~$S$
and the \emph{fppf} topology is replaced by the \'etale topology.
\end{remark}

\subsection{The subgroup of commutative extensions}

Given a group scheme $G$ over~$S$, let $\sigma_G\colon G^2\to G^2$ be
the involution switching the factors.  We have an obvious notion of
opposite group scheme $G^\op$ (replace the multiplication morphism
$m\colon G^2\to G$ by $m\circ\sigma_G$).  Given an extension
\[
\ses{F}{E}{G}
\]
we obtain a corresponding extension
\[
\ses{F}{E^\op}{G^\op}.
\]
Let $(T,\tau)$ be an $F$-extension datum on~$G$.  There is a canonical
isomorphism $\delta^1_{G^\op} T\isom \sigma_G^*(\delta^1T)$, where we
write $\delta^1_{G^\op}$ for the functor $\delta^1$ associated with
$G^\op$.  Let $\tau^\op\colon F_{G^2}\isom \delta^1_{G^\op} T$ be the
isomorphism making the diagram
\[
\xymatrix{F_{G^2} \ar[r]^{\tau^\op} \ar[d]^\sim& \delta^1_{G^\op} T \ar[d]^\sim\\
\sigma_G^*F_{G^2} \ar[r]^{\sigma_G^*\tau}& \sigma_G^*(\delta^1T)}
\]
commutative.  If $(T,\tau)$ defines the extension $E$ of~$G$, then
$(T,\tau^\op)$ defines the extension $E^\op$ of~$G^\op$.

Now suppose that $G$ is commutative, so $G=G^\op$.  Then $E$ is
commutative if and only if $\tau^\op=\tau$.
We use this to compute the subgroup $\Ext^1_S(G,F)$ of $\Ext_S(G,F)$
as follows.  We have a group homomorphism
\begin{equation}
\label{eq:Sigma}
\Sigma\colon\Ext_S(G,F)\longrightarrow F(G^2)
\end{equation}
sending the extension class defined by an extension datum $(T,\tau)$
to the ``commutator section'' $\Sigma(T,\tau)\in F(G^2)$ such that the
composed isomorphism
\[
F_{G^2}\morphism{\tau^\op}\delta^1 T\morphism{\tau^{-1}} F_{G^2}
\]
equals multiplication by $\Sigma(T,\tau)$.  Then $\Ext^1_S(G,F)$ is
the kernel of~$\Sigma$.

\subsection{Some results on $\bmu_n$-extension data}

\label{mu-extensions}

Let $n$ be a positive integer, and let $\bmu_n$ be the group scheme of
$n$-th roots of unity.  We now collect some results on $\bmu_n$-torsors
and central extensions by~$\bmu_n$ that will be used in
\S\ref{computing-H1}.

The groupoid $\Tors_{\bmu_n}(X)$ of $\bmu_n$-torsors over a scheme~$X$
is canonically equivalent to the following groupoid.  The objects are
pairs $(T,\lambda)$ where $T$ is a $\Gm$-torsor on~$X$ and
$\lambda\colon\Gmover{X}\isom T^{\otimes n}$ is an isomorphism of
$\Gm$-torsors.  The isomorphisms from $(T,\lambda)$ to $(T,\lambda')$
are the isomorphisms $\alpha\colon T\isom T'$ of $\Gm$-torsors
satisfying $\alpha^{\otimes n}\circ\lambda = \lambda'$.  The canonical
functor $\Tors_{\bmu_n}(X)\to\Tors_{\Gm}(X)$ obtained from the
inclusion $\bmu_n\to\Gm$ is given by sending $(T,\lambda)$ to~$T$.

Given a group scheme~$G$ over a scheme~$S$, a $\bmu_n$-extension datum
on~$G$ therefore consists of a $\bmu_n$-torsor $(T,\lambda)$ on~$G$ and
a trivialisation
\[
\tau\colon\Gmover{G}\isom\delta^1 T
\]
of $\Gm$-torsors such that the diagram
\[
\xymatrix{
\Gmover{G^2}\ar[r]^{\tau^{\otimes n}} \ar[d]_\sim&
(\delta^1 T)^{\otimes n} \ar[d]^\sim \\
\delta^1\Gmover{G}\ar[r]_{\delta^1\lambda}& \delta^1(T^{\otimes n})}
\]
commutes.

\begin{lemma}
\label{lemma:seq-mu}
There is a short exact sequence of Abelian groups
\[
1\longrightarrow G^*(S)/G^*(S)^n\longrightarrow
\Ext_S(G,\bmu_n)\longrightarrow\Ext_S(G,\Gm)[n]
\longrightarrow1,
\]
and similarly with $\Ext^1$ replaced by $\Ext$.
\end{lemma}

\begin{proof}
We construct a sequence as follows.  Representing $\bmu_n$-torsors as
above, we define a map $G^*(S)\to\Ext_S(G,\Gm)$ sending an element
$\lambda\in G^*(S)$ to the class of
$((\Gmover{G},\tilde\lambda),\tau_0)$, where the isomorphism
$\tilde\lambda\colon\Gmover{G}\isom\Gmover{G}^{\otimes n}$ is
multiplication by $\lambda$ (viewing $\lambda$ as an element of
$\Gm(G)$ and identifying $\Gmover{G}^{\otimes n}$ with~$\Gmover{G}$)
and $\tau_0$ is the canonical isomorphism
$\Gmover{G^2}\isom\delta^1\Gmover{G}$.  Furthermore, we define a map
$\Ext_S(G,\bmu_n)\to\Ext_S(G,\Gm)$ by sending $((T,\lambda),\tau)$ to
$(T,\tau)$.  One now verifies that this gives the desired short exact
sequence.
\end{proof}

\begin{remark}
Short exact sequences analogous to those in Lemma~\ref{lemma:seq-mu}
can be constructed from the long exact sequences obtained by applying
suitable derived functors to the Kummer sequence
\[
1\longrightarrow\bmu_n\longrightarrow\Gm\morphism{n}
\Gm\longrightarrow1
\]
on~$S_\fppf$.  An argument of Demazure and Gabriel
\cite[III.6.1.10]{Demazure-Gabriel} shows that these agree with the
exact sequences from Lemma~\ref{lemma:seq-mu}, at least up to a sign.
\end{remark}

\section{From extension data to $G^*$-torsors}

\label{sec:Ext-to-H1}

Let $G$ be a finite locally free commutative group scheme over a
scheme~$S$, and let $G^*$ denote its Cartier dual.  By a theorem of
Chase \cite[Theorem 16.14]{Chase-Sweedler}, generalised by
Shatz~\cite{Shatz} and Waterhouse~\cite{Waterhouse}, there is a
canonical isomorphism
\begin{equation}
\label{eq:iso-Chase}
\H^1(S_\fppf,G^*)\isom\Ext^1_S(G,\Gm).
\end{equation}

The explicit description of $\Ext^1_S(G,\Gm)$ given in the previous
section leads to the following explicit description of $G^*$-torsors.
For simplicity, we describe the case where $S$ is affine, say $S=\Spec
R$.  Then $G$ and $G^*$ are also affine, say
\[
G=\Spec B\quad\text{and}\quad G^*=\Spec B^\vee
\]
where $B$ is a finite locally free commutative and cocommutative Hopf
algebra over~$R$ and
\[
B^\vee=\Hom_{R\text{-Mod}}(B,R)
\]
is the Hopf algebra dual to~$B$.  We write $\mu$ for the
comultiplication map $B\to B\otimes_R B$.  Furthermore, $\Gm$-torsors
on~$G$ correspond to invertible $B$-modules, which are locally trivial
for the Zariski topology.  In particular, we may identify
$\H^1(G_\fppf,\Gm)$ with the Picard group $\Pic G$ of invertible
$B$-modules.

Consider a $\Gm$-extension datum $(U,\tau)$ on~$G$ defining a
commutative extension, where $B$ is now an invertible $B$-module and
$\tau$ is a trivialisation (given by a generating section, for
example) of the invertible $(B\otimes_R B)$-module
\[
(U\otimes_R B)\tensor{B\otimes_R B}(\mu^* U)^\vee
\tensor{B\otimes_R B}(B\otimes_R U)
\cong
(U\otimes_R U)\tensor{B\otimes_R B}(\mu^* U)^\vee.
\]
The morphism $m_\tau$ from~\eqref{eq:m-tau} corresponds to an
$R$-algebra homomorphism
\[
\mu_\tau\colon U\to U\otimes_R U
\]
satisfying $\mu_\tau(bu)=\mu(b)\mu_\tau(u)$ for all $b\in B$ and $u\in
U$.  Following Chase's construction in \cite[proof of
  Theorem~16.14]{Chase-Sweedler}, we obtain the following description
of the $G^*$-torsor corresponding to $(U,\tau)$.  The finite locally
free $R$-module
\[
U^\vee=\Hom_{R\text{-Mod}}(U,R)
\]
equipped with the $R$-bilinear map $U^\vee\times U^\vee\to U^\vee$
obtained by dualising $\mu_\tau$ is a commutative $R$-algebra, and the
$R$-linear map
\begin{equation}
\label{eq:alpha}
\alpha\colon U^\vee\to B^\vee\otimes_R U^\vee
\end{equation}
obtained by dualising the $B$-module structure on~$U$ defines a
$B^\vee$ comodule structure on~$U^\vee$.  The corresponding $S$-scheme
$X=\Spec U^\vee$ together with the morphism $\Spec\alpha\colon
G^*\times X\to X$ is then the desired $G^*$-torsor.

\section{Computational aspects}

\label{sec:computational}

We will now outline how the methods of this article can be used to do
explicit calculations with extensions and torsors under the assumption
that we can represent and compute with various more basic objects; see
Assumption~\ref{assumptions} below.  In \S\ref{localised-orders}, we
show that these assumptions are fulfilled for finite locally free
commutative group schemes over a localised order in a product of
number fields.

The algorithms described below have been partially implemented as part
of the author's software package \cite{dual-pairs-package} for
computing with finite group schemes in SageMath \cite{sagemath}.

\subsection{Presentations of finitely generated Abelian groups}

We briefly describe the tools that we will use for computing with
finitely generated Abelian groups; see Cohen \cite[\S4.1]{CohenAdv}
for details.

Let $A$ be a finitely generated Abelian group.  We assume that we have
a way of computationally representing elements of~$A$ and performing
the multiplication and inversion in~$A$.  (We allow for the
possibility that an element of~$A$ has several different computational
representions.)  By a \emph{presentation of~$A$} we mean non-negative
integers $r$ and~$k$, integers $d_1,\ldots,d_k\ge2$ with $d_1\mid
d_2\mid\cdots\mid d_k$ together with mutually inverse group
isomorphisms
\[
\exp_A\colon B\isom A,\quad
\log_A\colon A\isom B
\]
given by algorithms, where $B=\Z^r\oplus\bigoplus_{i=1}^k \Z/d_i\Z$.
We view $\log_A$ as a discrete logarithm function for~$A$.  By an
\emph{algorithm for finding linear relations in~$A$} we mean an
algorithm that given $a_1,\ldots,a_n\in A$ outputs the kernel of the
group homomorphism $\Z^n\to A$ sending the $i$-th standard basis
element to $a_i$.  Note that having a presentation for~$A$ is
equivalent to having a finite set of generators for~$A$ together with
an algorithm for finding linear relations in~$A$.  Furthermore, if we
can find linear relations, then we can compare elements: two elements
$a,a'\in A$ are equal if and only if the homomorphism $\Z\to A$
sending $1$ to $a'a^{-1}$ is trivial.

Let $f\colon A\to A'$ be a homomorphism of finitely generated Abelian
groups as above.  Assume that we can evaluate $f$ using the given
computational representation of elements of $A$ and~$A'$.  If we have
presentations of $A$ and~$A'$, we can compute a matrix for~$f$ with
respect to these presentations using $\exp_A$ and $\log_{A'}$.  From
such a matrix, we can compute presentations for the kernel and
cokernel of~$f$.  Note that to compute a presentation for the kernel
of~$f$, we do not need a presentation for $A'$; it suffices to have a
presentation for~$A$ and an algorithm for finding linear relations
in~$A'$.  Similarly, to compute a presentation for the cokernel
of~$F$, it suffices to have a presentation for~$A'$ and a finite set
of generators of~$A$.

\subsection{Computing extension class groups}

\label{computing-Ext}

Let $G$ be a finite locally free group scheme over a scheme~$S$, and
let $F$ be a sheaf of Abelian groups on~$S$ that is representable,
$\fppf$ and affine over~$S$.

\begin{assumption}
\label{assumptions}
We make the following computational assumptions about the group
scheme~$G$ and the sheaf $F$:
\begin{itemize}
\item The groups $F(G^i)$ (for $i\in\{1,2,3\}$) and
  $\H^1(G^i_\fppf,F)$ (for $i\in\{1,2\}$) are finitely generated.
\item We have computational representations for elements of $F(G)$,
  $F(G^2)$ and $F(G^3)$, and we can perform multiplication and
  inversion in these groups.
\item We have a finite set of generators for $F(G)$, a presentation of
  $F(G^2)$ and an algorithm for finding linear relations in $F(G^3)$.
\item We have computational representations for $F$-torsors on $G$,
  $G^2$ and $G^3$, and for isomorphisms between such torsors.
\item Given two $F$-torsors $T,T'$, we can compute $T\otimes T'$, and
  given trivialisations of $T$ and~$T'$, we can compute the resulting
  trivialisation of $T\otimes T'$; similarly for dual torsors.
\item Given an $F$-torsor $T$ on $G^2$ that is known to be trivial, we
  can find a trivialisation $F_{G^2}\isom T$.
\item Given an $F$-torsor $T$ on $G^3$ and an $F$-torsor automorphism
  $f\colon T\isom T$, we can find the unique element $u_f\in F(G^3)$
  such that $f$ equals multiplication by $u_f$.
\item We have a presentation for $\H^1(G_\fppf,F)$, and we can find
  linear relations in $\H^1(G^2_\fppf,F)$, using the given
  computational representation of $F$-torsors to represent elements of
  these groups.
\item We have algorithms for computing the various group homomorphisms
  and functors defined in \S\ref{simplicial}.
\end{itemize}
\end{assumption}

We use these assumptions and the exact sequence from
Theorem~\ref{th:exact-seq} to compute a presentation for $\Ext_S(G,F)$
as follows:
\begin{itemize}
\item Compute a presentation for $\HH^2(G,F)$ as the second cohomology
  group of the complex \eqref{eq:hochschild-complex}.
\item Compute a matrix for the homomorphism
  $d^1\colon\H^1(G_\fppf,F)\to\H^1(G^2_\fppf,F)$.
\item Compute a presentation for the group $K(G,F)=\ker d^1$.
\item Compute a matrix for the homomorphism $\xi_{G,F}$ from
  \eqref{eq:trg}.
\item Use Cohen's algorithm for computing a presentation for the
  second term in a left four-term exact sequence
  \cite[\S4.1.7]{CohenAdv} to compute a presentation for
  $\Ext_S(G,F)$.
\end{itemize}
In the last step, we use the description of the map
$\HH^2(G,F)\to\Ext_S(G,F)$ given in \S\ref{equiv-exact} to map
elements of $\HH^2(G,F)$ to $F$-extension data, and we use the
construction in Theorem~\ref{th:exact-seq} to lift elements of the
kernel of~$\xi_{G,F}$ to $F$-extension data.

We note that after computing $\Ext_S(G,F)$, we can also compute the
homomorphism \eqref{eq:Sigma} and its kernel, which is the group
$\Ext^1_S(G,F)$ of commutative extensions of $G$ by~$F$.

\begin{remark}
\label{rem:scalar}
For each $n\in\Z$, let $[n]\colon G\to G$ denote the
multiplication-by-$n$ map.  The kernel $K(G,F)$ of
$d^1\colon\H^1(G_\fppf,F)\to\H^1(G^2_\fppf,F)$ is contained in the
subgroup $\H^1(G_\fppf,F)^{(1)}$ of isomorphism classes of torsors $T$
such that for all $n\in\Z$ the torsors $[n]^* T$ and $T^{\otimes n}$
are isomorphic.  In practice, it may be useful to compute
$\H^1(G_\fppf,F)^{(1)}$ first and then to compute $K(G,F)$ as the
kernel of the restriction of $d^1$ to $\H^1(G_\fppf,F)^{(1)}$.  An
analogous remark in the context of Galois modules annihilated by a
prime number~$p$ was made by Schaefer and Stoll
\cite[Corollary~5.3]{Schaefer-Stoll}, who used this in their algorithm
for computing $p$-Selmer groups of elliptic curves.
\end{remark}

\subsection{Computing torsor class groups}

\label{computing-H1}

Let $G$ be a finite locally free and commutative group scheme over a
scheme~$S$.  We now consider the problem of computing the group
$\H^1(S_\fppf,G^*)$ of isomorphism classes of $G^*$-torsors.  In light
of the isomorphism~\eqref{eq:iso-Chase} between this group and
$\Ext^1_S(G,\Gm)$, it is natural to represent a $G^*$-torsor over~$S$
by the corresponding $\Gm$-extension datum on~$G$, and to view the
actual $G^*$-torsor (the $S$-scheme with $G^*$-action) as a
``secondary'' object to be computed from the $\Gm$-extension datum.

For simplicity, as in \S\ref{sec:Ext-to-H1}, we assume $S=\Spec R$,
$G=\Spec B$ and $G^*=\Spec B^\vee$ with $R$ a commutative ring and $B$
a finite locally free commutative cocommutative Hopf algebra over~$R$.
A $\Gm$-extension datum on~$G$ is therefore of the form $(U,\tau)$
where $U$ is an invertible $B$-module.  Using the description in
\S\ref{sec:Ext-to-H1}, the $R$-algebra structure on $U^\vee$ and the
comultiplication map \eqref{eq:alpha} can be extracted from $(U,\tau)$
using linear algebra over~$R$.

\begin{remark}
This representation of $G^*$-torsors fits very naturally into the
author's framework of \emph{dual pair of algebras} for computing with
finite group schemes \cite{dual-pairs}.  In this setting, neither the
comultiplication map~$\mu$ nor the comodule map $\alpha$ needs to be
written down explicitly.  This allows efficient computation with
$G^*$-torsors once $\H^1(S_\fppf,G^*)$ has been computed using one of
the methods described below.
\end{remark}

We now sketch two algorithms: one for computing $\H^1(S_\fppf,G^*)$,
and another for computing $\H^1(S_\fppf,G^*)[n]$ for a given positive
integer~$n$.  We assume that $S$ and~$G$ are such that our
computational assumptions \ref{assumptions} hold for the sheaf~$\Gm$
(for the first algorithm) and for the sheaf~$\bmu_n$ (for the second
algorithm).  For suitable rings~$R$, namely localised orders, this
will be justified in \S\ref{localised-orders} below.

\subsubsection*{Computing torsor class groups from $\Gm$-extensions}

The first method proceeds directly via the identification of
$\H^1(S_\fppf,G^*)$ with $\Ext^1_S(G,\Gm)$, and is conceptually more
straightforward than the method described below.  The algorithm is
simply to compute a presentation for $\Ext^1_S(G,\Gm)$ using the
algorithm from \S\ref{computing-Ext}, and then to compute, for each
extension datum $(U,\tau)$ in some finite generating set, the
resulting $R$-algebra structure on $U^\vee$ and the comodule map
$\alpha\colon U^\vee\to B^\vee\otimes_R U^\vee$.

\subsubsection*{Computing torsor class groups from $\bmu_n$-extensions}

In the second method, we replace $\Gm$ by $\bmu_n$, where $n$ is a
positive integer; this leads to an algorithm for computing the
$n$-torsion subgroup of $\Ext^1_S(G,\Gm)$ and hence of
$\H^1(S_\fppf,G^*)$.  The case where $n$ is (a divisor of) the
exponent of~$G$ is the most interesting in practice, but we do not
need this assumption.

In this approach, we first compute $\Ext^1_S(G,\bmu_n)$ using the
algorithm from \S\ref{computing-Ext}, and compute $\Ext^1_S(G,\Gm)[n]$
as the cokernel of the map $G^*(S)\to\Ext^1_S(G,\bmu_n)$ from
Lemma~\ref{lemma:seq-mu}.  We then proceed as in the first method,
using the isomorphism~\ref{eq:iso-Chase} to identify
$\Ext^1_S(G,\Gm)[n]$ with $\H^1(S_\fppf,G^*)[n]$.

\subsection{Comparison to algorithms for computing Selmer groups}

Part of the motivation behind the present work was to understand the
geometry behind existing algorithms for computing Selmer groups.  We
sketch briefly how these algorithms can be interpreted in the
framework described in this paper.

Let $E$ be an elliptic curve over a number field~$K$.  A standard way
of determining the Mordell--Weil group $E(K)$ starts by computing the
\emph{$n$-Selmer group} $\Sel_n(E)$ of the Galois cohomology group
$\H^1(K,E[n])$ for some $n\ge2$ (or more generally the Selmer group
associated with an isogeny).  Algorithms for computing these Selmer
groups were given by Schaefer and Stoll \cite{Schaefer-Stoll} (for $n$
prime) and by Cremona, Fisher, O'Neil, Simon and Stoll \cite{CFOSS-I},
among others.  These algorithms are based on mapping $\Sel_n(E)$ to a
subgroup of the Galois cohomology group $\H^1(R,\bmu_n)\cong
R^\times/(R^\times)^n$ for a certain \'etale $K$-algebra $R$.  This in
turn uses the embedding of $E[n]$ into the Galois module of functions
$E[n]\to\bmu_n$ defined by the Weil pairing; see
\cite[\S3]{Schaefer-Stoll} and \cite[\S3]{CFOSS-I}.  In
\cite[\S1.5]{CFOSS-I} it was noted that the group $\H^1(K,E[n])$
classifying $E[n]$-torsors also classifies commutative extensions of
$E[n]$ by~$\Gm$.  This point of view was used in \cite[\S3]{CFOSS-I}
to identify $\Sel_n(E)$ as a subquotient of $(R\otimes_K R)^\times$.

In this paper, we consider group schemes over more general base
schemes and use \emph{fppf} cohomology instead of Galois cohomology
together with local conditions.  The link between the two approaches
is that Selmer groups of Abelian varieties can be interpreted as
\emph{fppf} cohomology groups, as shown by \v{C}esnavi\v{c}ius
\cite[\S4]{Cesnavicius}.  Computing the $n$-Selmer group of an
elliptic curve over a number field~$K$ can therefore be viewed as
computing $\H^1(S_\fppf,E[n])$, with $S$ the spectrum of the ring of
$\Sigma$-integers in~$K$ for a finite set $\Sigma$ of places of~$K$,
followed by computing a subgroup defined by local conditions at the
places in~$\Sigma$.

Of the two methods given in \S\ref{computing-H1} for computing
$\H^1(S_\fppf,G^*)$ (note that if $G$ is the $n$-torsion of an
elliptic curve, then we can identify $G$ with $G^*$ via the Weil
pairing), the second method is closest to the algorithms of
\cite{Schaefer-Stoll} and ~\cite{CFOSS-I}.  This second method also
has certain (potential) practical advantages over the first:
\begin{enumerate}
\item Computing presentations for the groups $\bmu_n(G^i)$ is easier
  than for $\Gm(G^i)$, because one only needs to know the $n$-th roots
  of unity rather than the full unit groups of the algebras in
  question.  The same holds for finding linear relations in
  $\HH^3(G,\bmu_n)$ as opposed to $\HH^3(G,\Gm)$.
\item At least in certain cases, it may be easier to compute the
  subgroup $K(G,\bmu_n)$ of $\H^1(G_\fppf,\bmu_n)$ than to compute the
  subgroup $K(G,\Gm)$ of $\H^1(G_\fppf,\Gm)$.  In the case where $p$
  is an odd prime number and $E$ is an elliptic curve over a number
  field~$K$, Schaefer and Stoll \cite[\S5]{Schaefer-Stoll} showed that
  the Galois cohomology group $\H^1(K,E[p])$ and the $p$-Selmer group
  of~$E$ can be computed as certain subgroups of the kernel of a
  homomorphism $A^\times/(A^\times)^p\to B^\times/(B^\times)^p$, where
  $A$ and~$B$ are $K$-algebras of degree $p^2-1$.  Translating this to
  our setting, and taking $n$ to be an odd prime number~$p$ and $G$ to
  be a group scheme over~$S$ annihilated by~$p$, we may wonder if
  $K(G,\bmu_p)$ can similarly be computed as the kernel of a
  homomorphism $\H^1(G_\fppf,\bmu_p)^{(1)}\to\H^1(X_\fppf,\bmu_p)$ (see
  Remark~\ref{rem:scalar} for the definition of the left-hand side)
  for a suitable $S$-scheme $X$ of substantially smaller degree than
  that of $G^2$.
\end{enumerate}

\subsection{Results over localised orders}

\label{localised-orders}

We conclude by showing how the algorithms from this paper can be
implemented concretely for suitable base schemes, based on the
computation of unit groups and Picard groups of (localisations of)
orders in number fields.

\begin{definition}
A (reduced) \emph{order} is a reduced commutative ring that is free of
finite rank as a $\Z$-module.
\end{definition}

An order is in particular Noetherian and one-dimensional, but not
necessarily regular, and is of finite index in a product of maximal
orders of number fields.

\begin{definition}
A \emph{localised order} is a ring of the form $R_\Sigma$, where $R$
is an order, $\Sigma$ is a finite set of maximal ideals of~$R$, and
$R_\Sigma$ is the coordinate ring of the complement of $\Sigma$ in
$\Spec R$.
\end{definition}

\begin{example}
Let $K$ be a number field, and let $\Sigma$ be a finite set of places
of~$K$.  Then the ring $\Z_{K,\Sigma}$ of $\Sigma$-integers in~$K$ is
a localised order.
\end{example}

Let $R$ be a localised order, and let $G$ be a finite locally free
group scheme over $S=\Spec R$.  Then each $G^i$ is the spectrum of a
finite locally free $R$-algebra $B^i$.  Furthermore, $G^i$ is
generically \'etale over~$S$, so $B^i$ is again a localised order.

There are algorithms for computing presentations of unit groups and
Picard groups of orders; see Cohen \cite[\S6.5]{Cohen} for maximal
orders in number fields, Kl\"uners and Pauli \cite{Klueners-Pauli} for
general orders in number fields and Marseglia
\cite[Remark~3.8]{Marseglia} for arbitrary orders.  These algorithms
can be extended to localised orders as in \cite[\S7.4]{CohenAdv}.  If
$R'$ is a localised order, we represent $\Gm$-torsors (or invertible
sheaves) on~$\Spec R'$ by invertible fractional ideals of~$R'$.

Similarly, if $R'$ is a localised order and $n$ is a positive integer,
then as in \S\ref{mu-extensions} we represent $\bmu_n$-torsors over
$S'=\Spec R'$ by pairs $(J,x)$ where $J$ is a fractional ideal of~$R'$
and $x$ is a generator of $J^n$.  By the long exact cohomology
sequence obtained from the Kummer sequence, the group
$\H^1(S'_\fppf,\bmu_n)$ of isomorphism classes of $\bmu_n$-torsors fits
in a short exact sequence
\[
1\longrightarrow R'^\times/(R'^\times)^n\longrightarrow
\H^1(S'_\fppf,\bmu_n)\longrightarrow(\Pic R')[n]\longrightarrow1,
\]
which we can use to compute a presentation of $\H^1(S'_\fppf,\bmu_n)$.

The above implies that if $R$ is a localised order and $G$ is a finite
locally free group scheme over~$S=\Spec R$, then our computational
assumptions \ref{assumptions} are fulfilled both for the sheaf~$\Gm$
and for the sheaf~$\bmu_n$.  We can therefore apply the method from
\eqref{computing-Ext} and both methods from \S\ref{computing-H1} to
compute presentations for the groups $\Ext_S(G,\Gm)$,
$\Ext^1_S(G,\Gm)$ and $\H^1(S_\fppf,G^*)$ (in the case of the second
method, for the $n$-torsion of these groups).

Finally, we consider two finite locally free commutative group schemes
$G$ and~$F$ over~$S$.  Then we can compute $\Ext_S(G,F)$ using the
following ``bootstrap'' argument.  We can compute presentations for
the finite Abelian groups $F(G^i)$ for $i\in\{1,2,3\}$; this comes
down to computing homomorphisms between subrings of products of number
fields.  Furthermore, we can compute $\H^1(G^i_\fppf,F)$ for $i=1$ and
$i=2$ as described above (with $(G^i,F^*)$ in place of $(S,G)$)
because the $G^i$ are again spectra of localised orders.  Finally,
using the representation of $F$-torsors as $\Gm$-extension data allows
us to perform the remaining tasks in Assumption~\ref{assumptions}.
Therefore our computational assumptions are fulfilled for the group
scheme~$G$ over~$S$ and the sheaf $F$, and we can use the algorithm
from \S\ref{computing-Ext} to compute the group $\Ext_S(G,F)$.

\bibliographystyle{alpha}
\bibliography{biblio}

\end{document}